\documentclass[11pt,reqno]{amsart}

\usepackage[T1]{fontenc}
\usepackage[utf8]{inputenc}
\usepackage{lmodern}
\usepackage{microtype}
\usepackage[a4paper,margin=1.15in]{geometry}
\usepackage{amsmath,amssymb,amsthm,mathtools}
\usepackage{mathrsfs}
\usepackage{enumitem}
\usepackage{aliascnt}
\usepackage{xcolor}
\usepackage{hyperref}
\usepackage[nameinlink,capitalize,noabbrev]{cleveref}

\DeclareFontFamily{U}{mathx}{\hyphenchar\font45}
\DeclareFontShape{U}{mathx}{m}{n}{
  <5> <6> <7> <8> <9> <10>
  <10.95> <12> <14.4> <17.28> <20.74> <24.88>
  mathx10
}{}
\DeclareSymbolFont{mathx}{U}{mathx}{m}{n}
\DeclareFontSubstitution{U}{mathx}{m}{n}
\DeclareMathAccent{\widecheck}{0}{mathx}{"71}

\definecolor{linkblue}{RGB}{20,63,110}
\definecolor{citegreen}{RGB}{35,92,67}
\hypersetup{
  colorlinks=true,
  linkcolor=linkblue,
  citecolor=citegreen,
  urlcolor=linkblue,
  pdfauthor={Nobuo Iida},
  pdftitle={The Three-Dimensional Symplectization Question for Small Seifert L-Spaces and Further Classes},
  pdfsubject={Contact and symplectic topology},
  pdfkeywords={symplectization, contact structure, small Seifert fibered space, L-space, strong fillability, connected sum, S1 times S2, figure-eight knot surgery, Brieskorn sphere, Whitehead link surgery, Liouville cobordism, monopole Floer contact invariant}
}

\numberwithin{equation}{section}
\allowdisplaybreaks[1]
\newtheorem{theorem}{Theorem}[section]

\newaliascnt{proposition}{theorem}
\newtheorem{proposition}[proposition]{Proposition}
\aliascntresetthe{proposition}

\newaliascnt{lemma}{theorem}
\newtheorem{lemma}[lemma]{Lemma}
\aliascntresetthe{lemma}

\newaliascnt{corollary}{theorem}
\newtheorem{corollary}[corollary]{Corollary}
\aliascntresetthe{corollary}

\theoremstyle{definition}
\newaliascnt{definition}{theorem}
\newtheorem{definition}[definition]{Definition}
\aliascntresetthe{definition}

\newaliascnt{question}{theorem}
\newtheorem{question}[question]{Question}
\aliascntresetthe{question}
\crefname{question}{Question}{Questions}

\theoremstyle{remark}
\newaliascnt{remark}{theorem}
\newtheorem{remark}[remark]{Remark}
\aliascntresetthe{remark}

\newcommand{\R}{\mathbb{R}}
\newcommand{\Q}{\mathbb{Q}}
\newcommand{\Z}{\mathbb{Z}}
\newcommand{\F}{\mathbb{F}_2}
\newcommand{\Tplus}{\mathcal{T}^{+}}

\newcommand{\SpinC}{\operatorname{Spin}^{c}}
\newcommand{\congsymp}{\cong_{\mathrm{symp}}}
\newcommand{\congcont}{\cong_{\mathrm{cont}}}

\title[Symplectization question in dimension three]
{The Three-Dimensional Symplectization Question for Small Seifert L-Spaces and Further Classes}
\author{Nobuo Iida}
\address{Kavli Institute for the Physics and Mathematics of the Universe (WPI), The University of Tokyo, 5-1-5 Kashiwanoha, Kashiwa, Chiba 277-8583, Japan}
\email{iidanobuo1224@g.ecc.u-tokyo.ac.jp}
\date{August 25, 2026}

\begin{document}
\begin{abstract}

We solve the three-dimensional symplectization question for several
examples.  These include small Seifert fibered $L$-spaces, a family of
Whitehead-link surgery $L$-spaces, $S^1\times S^2$, certain surgeries
on the figure-eight knot, explicit families of Brieskorn spheres, and
arbitrary finite connected sums mixing these manifolds.  In particular,
the result covers the Weeks manifold and infinitely many hyperbolic
$L$-spaces.  The main symplectic step is to construct Liouville
cobordisms in both directions from a symplectomorphism of
symplectizations.  The resulting preservation of strong fillability
and of the vanishing of the monopole Floer contact invariant, together
with the comparison with the Heegaard Floer contact invariant,
connected-sum formulas, and classification results for tight contact
structures, proves the stated rigidity results.

\end{abstract}

\maketitle
\tableofcontents

\section{Introduction}

Let $Y$ be a closed connected oriented three-manifold, and let $\xi\subset TY$ be a positive cooriented contact structure.  For a positive contact form $\alpha$ with $\ker\alpha=\xi$, its symplectization is
\begin{equation}\label{eq:symplectization}
  S(Y,\alpha)
  =\bigl(\R_t\times Y,\lambda_\alpha=e^t\alpha,
  \omega_\alpha=d\lambda_\alpha\bigr).
\end{equation}
If $\alpha'=e^g\alpha$, then
\[
  (t,y)\longmapsto(t-g(y),y)
\]
pulls $e^t\alpha'$ back to $e^t\alpha$.  Thus the exact symplectic isomorphism type of the pair in \eqref{eq:symplectization} depends only on $\xi$, and we write $S(Y,\xi)$ when the defining form is irrelevant.

Courte constructed high-dimensional contact manifolds that are not contactomorphic, and even have nondiffeomorphic underlying manifolds, although their symplectizations are exact symplectomorphic \cite{Courte2014}.  In dimension three, the corresponding problem remains substantially more rigid.

\begin{question}[Three-dimensional symplectization question]\label{ques:symplectization-question}
Do there exist closed contact three-manifolds $(Y_0,\xi_0)$ and $(Y_1,\xi_1)$ that are not contactomorphic although
\[
  S(Y_0,\xi_0)\congsymp S(Y_1,\xi_1)?
\]
The symplectomorphism is not assumed to be exact.
\end{question}

This is Problem~40 in McDuff--Salamon \cite{McDuffSalamon2017}.  Throughout the paper, classes of closed oriented three-manifolds are understood up to orientation-preserving diffeomorphism.  This convention is harmless: pulling back a contact structure by such a diffeomorphism does not change either side of the symplectization question.

In \cite{Iida2026Overtwisted}, the author proved that an ordinary symplectomorphism of symplectizations determines an orientation-preserving diffeomorphism of the underlying three-manifolds that preserves the full homotopy class of the oriented contact plane field.  In particular, Eliashberg's classification settles \cref{ques:symplectization-question} affirmatively when both contact structures are overtwisted. Thus the remaining cases are those in which at least one of the two contact structures is tight, namely the tight--overtwisted and tight--tight cases.

We now describe all classes covered by the main theorem.  Let $\mathscr S$ denote the class of closed connected oriented $L$-spaces admitting Seifert fibrations over $S^2$ with at most three exceptional fibers.  Thus $\mathscr S$ includes both the usual small Seifert case of exactly three exceptional fibers and the lens-space cases with fewer exceptional fibers.

For the Whitehead link, let $M(n,r)$ denote the oriented three-manifold obtained by $(n,r)$-Dehn surgery on the ordered link, using the slope conventions of Min--Nonino \cite{MinNonino2026}, and set
\begin{equation}\label{eq:Rplus}
  \mathcal R_+
  =\bigl([2,4)\cup[5,\infty)\bigr)\cap\Q.
\end{equation}
Define
\begin{equation}\label{eq:whitehead-class}
  \mathscr W
  =\{M(n,r): n\ge5,\ r\in\mathcal R_+\}.
\end{equation}
These are $L$-spaces by Liu's calculation of the Whitehead-link $L$-space surgery region \cite[Proposition~6.4]{Liu2017}.

To avoid conflict with the notation $M(n,r)$, write
\[
  E(r):=S^3_r(4_1)
\]
for smooth $r$-surgery on the figure-eight knot, with the surgery orientation, and set
\begin{equation}\label{eq:figure-eight-class}
\begin{split}
  \mathscr E={}&\{E(n):n>0,\ n\ne4\}
  \cup\{E(-1),E(-2),E(-3)\}\\
  &\cup\{E(-1/m):m>1\}.
\end{split}
\end{equation}
For Brieskorn manifolds, all positively written $\Sigma(a_1,\ldots,a_k)$ below carry their canonical link orientation.  Define
\begin{equation}\label{eq:brieskorn-class}
\begin{aligned}
\mathscr B={}&
 \{-\Sigma(2,3,5),\ \Sigma(2,3,5),\ -\Sigma(2,3,7),\ -\Sigma(2,3,11)\}\\
&\cup\{\Sigma(2,3,6k+1):k\ge1\}
 \cup\{\Sigma(2,3,6k-1):k\ge2\}\\
&\cup\{\Sigma(2,5,7),\Sigma(2,5,9),\Sigma(3,4,5),
        \Sigma(3,7,10),\Sigma(2,7,11),\Sigma(3,7,19),\\
&\hspace{3.7em}\Sigma(3,4,11),\Sigma(3,8,11),\Sigma(3,5,7),
        \Sigma(4,5,9),\Sigma(3,5,14)\}.
\end{aligned}
\end{equation}
Let
\begin{equation}\label{eq:fillable-single-class}
  \mathscr F
  =\{S^1\times S^2\}\cup\mathscr E\cup\mathscr B
\end{equation}
and define the unified class
\begin{equation}\label{eq:unified-class}
  \mathscr U
  =\mathscr S\cup\mathscr W\cup\mathscr F.
\end{equation}
For any class $\mathscr D$ of closed connected oriented three-manifolds, write $\mathscr D^{\#}$ for the class of all nonempty finite connected sums of members of $\mathscr D$.  Factors diffeomorphic to $S^3$ may be omitted.

\begin{theorem}[Main theorem]\label{thm:main}
For $i=0,1$, let $Y_i\in\mathscr U^{\#}$, and let $\xi_i$ be a positive cooriented contact structure on $Y_i$.  Then
\begin{equation}\label{eq:main-equivalence}
  S(Y_0,\xi_0)\congsymp S(Y_1,\xi_1)
  \quad\Longleftrightarrow\quad
  (Y_0,\xi_0)\congcont(Y_1,\xi_1).
\end{equation}
The symplectomorphism on the left is ordinary and need not be exact.  A contactomorphism on the right lifts to a strict exact symplectomorphism.
\end{theorem}

\begin{corollary}[The Weeks manifold]\label{cor:weeks}
The symplectization question has an affirmative answer for arbitrary positive cooriented contact structures on the oriented Weeks manifold $M(5,5/2)$.
\end{corollary}

\begin{corollary}[An infinite hyperbolic family]\label{cor:hyperbolic}
There are infinitely many closed hyperbolic $L$-spaces for which the symplectization question has an affirmative answer for arbitrary positive cooriented contact structures.
\end{corollary}

The last statement concerns the indicated Whitehead-link surgery family; no assertion is made for arbitrary hyperbolic $L$-spaces.

The main new symplectic input is independent of all classification hypotheses.

\begin{theorem}[Bidirectional Liouville $h$-cobordisms]\label{thm:bidirectional-intro}
Let $(Y_i,\xi_i)$, $i=0,1$, be closed connected positive cooriented contact three-manifolds.  If
\[
  S(Y_0,\xi_0)\congsymp S(Y_1,\xi_1),
\]
then there are Liouville cobordisms
\[
  (Y_0,\xi_0)\rightsquigarrow(Y_1,\xi_1),
  \qquad
  (Y_1,\xi_1)\rightsquigarrow(Y_0,\xi_0).
\]
Their underlying smooth cobordisms may be chosen invertible and hence are $h$-cobordisms.
\end{theorem}

For a symplectomorphism $\Phi$, the transported primitive and the target primitive differ by a closed one-form.  On $\R\times Y$, every closed one-form is the sum of the pullback of a closed form on $Y$ and an exact form.  The pullback term need not be exact when $b_1(Y)>0$.  Nevertheless, on a sufficiently high slice $\{r\}\times Y$, with $r\gg0$,
the contact form $e^r\alpha$ dominates this fixed closed form.  Cutting off the exact term then produces a Liouville primitive on a compact smooth cobordism cut out by $\Phi$.  Applying the same construction to $\Phi^{-1}$ gives the reverse cobordism.

The same bidirectional cobordisms also show that strong fillability is preserved by symplectomorphic symplectizations.

\begin{corollary}[Strong fillability]\label{cor:strong-fillability-intro}
If two closed positive cooriented contact three-manifolds have symplectomorphic symplectizations, then one is strongly symplectically fillable if and only if the other is.
\end{corollary}

Floer theory supplies the second principal consequence.  Let $\widecheck c_{\mathrm{HM}}(\xi;\F)$ denote the monopole Floer contact invariant.  The bidirectional cobordisms and Echeverria's naturality theorem \cite{Echeverria2020} imply the following.

\begin{corollary}[Monopole contact invariant]\label{cor:HM-intro}
If two closed positive cooriented contact three-manifolds have symplectomorphic symplectizations, then
\[
  \widecheck c_{\mathrm{HM}}(\xi_0;\F)=0
  \quad\Longleftrightarrow\quad
  \widecheck c_{\mathrm{HM}}(\xi_1;\F)=0.
\]
\end{corollary}

The proof of \cref{thm:main} separates the factors in $\mathscr U$ into two types.  The $L$-space families $\mathscr S$ and $\mathscr W$ are controlled by the Heegaard Floer contact invariant: Matkovi\v{c} and Min--Nonino show that it detects tightness and distinguishes the relevant tight structures \cite{Matkovic2018,MinNonino2026}.  The remaining class $\mathscr F$ has at most one tight contactomorphism class on each underlying manifold, and every such tight structure is strongly fillable; this follows from the standard result for $S^1\times S^2$, Conway--Min's classification of figure-eight surgeries, and Cavallo's classification of the Brieskorn families \cite{ConwayMin2020,Cavallo2026}.

The Taubes--Colin--Ghiggini--Honda comparison identifies the plus Heegaard Floer contact invariant with the monopole Floer contact invariant \cite{TaubesECHSWV,ColinGhigginiHondaIII}.  Colin's decomposition theorem, the Heegaard Floer K\"unneth theorem, and the connected-sum formula for contact invariants propagate the $L$-space information across connected sums \cite{Colin1997,DingGeiges2007,OzsvathSzabo2004Properties,OzsvathSzabo2005}.  Strongly fillable factors are then added by a strong symplectic cobordism built from a filling and Weinstein one-handles; Echeverria's naturality propagates nonvanishing of the monopole contact invariant through this cobordism.  Together with formal rigidity, these ingredients prove the main theorem.

\subsection*{Organization of the paper}
Section~\ref{sec:preliminaries} records the conventions and the results from \cite{Iida2026Overtwisted} used here.  Section~\ref{sec:liouville} proves \cref{thm:bidirectional-intro}.  Section~\ref{sec:floer} derives the monopole and Heegaard Floer contact-invariant consequences.  Section~\ref{sec:criterion} develops the classwise criterion, the connected-sum formulas, and the mixed connected-sum argument.  Section~\ref{sec:strong} proves invariance of strong fillability, establishes a strong-fillability criterion, and verifies the required properties for $S^1\times S^2$, the figure-eight surgeries, and the Brieskorn families.  Section~\ref{sec:seifert} records Matkovi\v{c}'s classification for small Seifert $L$-spaces.  Section~\ref{sec:whitehead} records the Whitehead-link results of Liu and Min--Nonino and completes the proof of \cref{thm:main,cor:weeks,cor:hyperbolic}.

\section{Preliminaries}\label{sec:preliminaries}

All manifolds are smooth, Hausdorff, and second countable.  Contact structures are positive and cooriented, and contactomorphisms preserve the coorientation.  The symplectic orientation on a symplectization agrees with the product orientation because
\[
  \frac{1}{2}d(e^t\alpha)^2
  =e^{2t}\,dt\wedge\alpha\wedge d\alpha.
\]

\subsection{Strict lifts of contactomorphisms}

\begin{lemma}[Strict lift]\label{lem:strict-lift}
Let
\[
  f\colon(Y_0,\xi_0)\longrightarrow(Y_1,\xi_1)
\]
be a coorientation-preserving contactomorphism.  Choose positive contact forms $\alpha_i$ defining $\xi_i$, and write
\[
  f^*\alpha_1=e^g\alpha_0
\]
for a smooth function $g\colon Y_0\to\R$.  Then
\[
  \widetilde f(t,y)=(t-g(y),f(y))
\]
is a strict exact symplectomorphism satisfying
\[
  \widetilde f^*(e^t\alpha_1)=e^t\alpha_0.
\]
\end{lemma}

\begin{proof}
This follows from the direct calculation
\[
  \widetilde f^*(e^t\alpha_1)
  =e^{t-g(y)}f^*\alpha_1
  =e^t\alpha_0.
\]
Taking exterior derivatives proves the symplectic assertion.
\end{proof}

\subsection{Formal rigidity and control of the ends}

We recall three results from \cite{Iida2026Overtwisted}.  For a closed connected manifold $Y$, write
\[
  E_-^Y(T)=(-\infty,-T]\times Y,
  \qquad
  E_+^Y(T)=[T,\infty)\times Y.
\]
A proper map $F\colon\R\times Y_0\to\R\times Y_1$ is \emph{end-order-preserving} if, for every $R>0$, there exists $T>0$ such that
\[
  F(E_-^{Y_0}(T))\subset E_-^{Y_1}(R),
  \qquad
  F(E_+^{Y_0}(T))\subset E_+^{Y_1}(R).
\]

\begin{theorem}[Formal symplectization rigidity \cite{Iida2026Overtwisted}]\label{thm:formal-rigidity}
Let $(Y_i,\xi_i)$, $i=0,1$, be closed connected positive cooriented contact three-manifolds.  If
\[
  \Phi\colon S(Y_0,\xi_0)\longrightarrow S(Y_1,\xi_1)
\]
is a symplectomorphism, then there is an orientation-preserving diffeomorphism
\[
  f\colon Y_0\longrightarrow Y_1
\]
such that $\xi_0$ and $f^*\xi_1$ are homotopic as oriented plane fields.
\end{theorem}

\begin{proposition}[End control \cite{Iida2026Overtwisted}]\label{prop:end-control}
Every symplectomorphism between symplectizations of closed connected contact three-manifolds is end-order-preserving.  If an end-order-preserving diffeomorphism is written as
\[
  F(t,y)=(a(t,y),b(t,y)),
\]
then
\begin{equation}\label{eq:uniform-divergence}
  \lim_{t\to+\infty}\inf_{y\in Y_0}a(t,y)=+\infty,
  \qquad
  \lim_{t\to-\infty}\sup_{y\in Y_0}a(t,y)=-\infty.
\end{equation}
The same assertions hold for $F^{-1}$.
\end{proposition}

The proof distinguishes the two ends by symplectic volume: a negative half-symplectization has finite volume, whereas every positive half-symplectization has infinite volume.  We shall also use the following smooth-topological conclusion.  The statement is \cite[Theorem~5.2 and Proposition~5.3]{Iida2026Overtwisted}, based on Hausmann--Jahren's correspondence between open-product diffeomorphisms and invertible cobordisms \cite{HausmannJahren2018}.

\begin{proposition}[Smooth cuts are invertible]\label{prop:smooth-cut-input}
Let $F\colon\R\times Y_0\to\R\times Y_1$ be an orientation-preserving end-order-preserving diffeomorphism.  Suppose that $T,r\in\R$ satisfy
\begin{equation}\label{eq:smooth-cut-hypotheses}
  F(\{T\}\times Y_0)\subset(r,\infty)\times Y_1,
  \qquad
  (-\infty,r]\times Y_1\subset F(( -\infty,T)\times Y_0).
\end{equation}
Then
\[
  F(( -\infty,T]\times Y_0)\cap([r,\infty)\times Y_1)
\]
is a compact smooth cobordism from $\{r\}\times Y_1$ to $F(\{T\}\times Y_0)$.  With the boundary parametrizations induced by the two standard slices, it is an invertible smooth cobordism from $Y_1$ to $Y_0$.  In particular, it is an $h$-cobordism.
\end{proposition}

\subsection{Liouville cobordisms}

\begin{definition}\label{def:liouville-cobordism}
A \emph{Liouville cobordism} from $(Y_-,\xi_-)$ to $(Y_+,\xi_+)$ consists of a compact oriented cobordism $W$ with
\[
  \partial W=-\partial_-W\sqcup\partial_+W,
\]
boundary parametrizations
\[
  j_-\colon Y_-\longrightarrow\partial_-W,
  \qquad
  j_+\colon Y_+\longrightarrow\partial_+W,
\]
and a one-form $\Lambda\in\Omega^1(W)$ such that $d\Lambda$ is symplectic, the pullbacks $j_-^*\Lambda$ and $j_+^*\Lambda$ are positive contact forms for $\xi_-$ and $\xi_+$, respectively, and the Liouville vector field $Z$, defined by
\[
  \iota_Zd\Lambda=\Lambda,
\]
points inward along $\partial_-W$ and outward along $\partial_+W$.
\end{definition}

A \emph{Liouville $h$-cobordism} is a Liouville cobordism whose underlying smooth cobordism is an $h$-cobordism.

\section{Liouville cobordisms cut out by a symplectomorphism}\label{sec:liouville}

Let
\[
  \Phi\colon
  \bigl(\R\times Y_0,d(e^t\alpha_0)\bigr)
  \longrightarrow
  \bigl(\R\times Y_1,d(e^t\alpha_1)\bigr)
\]
be a symplectomorphism, where $\alpha_i$ is a positive contact form for $\xi_i$.  Set
\[
  V_i=\R\times Y_i,
  \qquad
  \lambda_i=e^t\alpha_i,
  \qquad
  \omega_i=d\lambda_i.
\]

\subsection{The closed one-form discrepancy}

On $V_1$, define the transported primitive
\begin{equation}\label{eq:mu}
  \mu=(\Phi^{-1})^*\lambda_0.
\end{equation}
Since $d\mu=\omega_1=d\lambda_1$, the one-form
\[
  \delta=\mu-\lambda_1
\]
is closed.

\begin{lemma}[Product decomposition of the discrepancy]\label{lem:discrepancy}
Let $\pi\colon\R\times Y_1\to Y_1$ be projection.  There are a closed one-form $\beta\in\Omega^1(Y_1)$ and a smooth function $G\colon V_1\to\R$ such that
\begin{equation}\label{eq:discrepancy-decomposition}
  \mu-\lambda_1=\pi^*\beta+dG.
\end{equation}
\end{lemma}

\begin{proof}
Let $i\colon Y_1\hookrightarrow\R\times Y_1$ be the zero slice and set $\beta=i^*\delta$.  Then $\beta$ is closed.  Since $i\circ\pi$ is homotopic to the identity, $\delta$ and $\pi^*\beta$ represent the same de Rham cohomology class.  Their difference is therefore exact.
\end{proof}

\begin{lemma}[High-slice contact domination]\label{lem:high-slice}
For all sufficiently large $r$, the one-forms
\begin{equation}\label{eq:theta}
  \theta_{r,s}=e^r\alpha_1+s\beta,
  \qquad 0\le s\le1,
\end{equation}
are positive contact forms on $Y_1$.  Moreover,
\begin{equation}\label{eq:positive-radial-component}
  1+e^{-r}\beta(R_{\alpha_1})>0
\end{equation}
pointwise on $Y_1$.
\end{lemma}

\begin{proof}
Because $d\beta=0$,
\[
  \theta_{r,s}\wedge d\theta_{r,s}
  =e^{2r}\alpha_1\wedge d\alpha_1
   +s e^r\beta\wedge d\alpha_1.
\]
After division by $e^{2r}$, the second term converges uniformly to zero as $r\to+\infty$, uniformly in $s\in[0,1]$, whereas $\alpha_1\wedge d\alpha_1$ is a positive volume form.  This proves the first assertion.  The second follows because $Y_1$ is compact, and hence the smooth
function $\beta(R_{\alpha_1})$ is bounded.
\end{proof}

\subsection{The compact smooth cut}

Write
\[
  \Phi(t,y)=(a(t,y),b(t,y)),
  \qquad
  \Phi^{-1}(s,z)=(\bar a(s,z),\bar b(s,z)).
\]

\begin{lemma}[Choice of the two boundary levels]\label{lem:choice-levels}
Fix $r\in\R$.  There are $T\in\R$ and $\varepsilon>0$ such that
\begin{align}
  \Sigma_T:=\Phi(\{T\}\times Y_0)
  &\subset(r+\varepsilon,\infty)\times Y_1,
  \label{eq:upper-above}\\
  (-\infty,r+\varepsilon]\times Y_1
  &\subset\Phi(( -\infty,T)\times Y_0).
  \label{eq:lower-contained}
\end{align}
The set
\begin{equation}\label{eq:W10}
  W_{10}
  =\Phi(( -\infty,T]\times Y_0)
   \cap([r,\infty)\times Y_1)
\end{equation}
is a compact smooth cobordism from the standard slice $\{r\}\times Y_1$ to $\Sigma_T$.  With the natural boundary parametrizations, its underlying smooth cobordism from $Y_1$ to $Y_0$ is invertible and hence an $h$-cobordism.
\end{lemma}

\begin{proof}
By the negative-end statement in \eqref{eq:uniform-divergence}, applied to $\Phi^{-1}$, there is $s_0<r$ such that
\[
  \bar a(s,z)<0
  \qquad(s\le s_0,\ z\in Y_1).
\]
The function $\bar a$ is bounded above on the compact set $[s_0,r]\times Y_1$.  Consequently
\begin{equation}\label{eq:Mr}
  M_r:=\sup\{\bar a(s,z):s\le r,\ z\in Y_1\}<\infty.
\end{equation}
By \eqref{eq:uniform-divergence}, we can choose $T>M_r$ so large that
\[
  m_T:=\min_{y\in Y_0}a(T,y)>r.
\]
Set $\varepsilon_1=m_T-r>0$.  Since
\[
  \max_{z\in Y_1}\bar a(r,z)\le M_r<T,
\]
uniform continuity on a compact neighborhood of $\{r\}\times Y_1$ gives $\varepsilon_2>0$ such that
\[
  \bar a(s,z)<T
  \qquad(r\le s\le r+\varepsilon_2,\ z\in Y_1).
\]
Choose
\[
  0<\varepsilon<\min\{\varepsilon_1,\varepsilon_2\}.
\]
Then $a(T,y)>r+\varepsilon$ for every $y$, proving \eqref{eq:upper-above}.  If $s\le r$, \eqref{eq:Mr} and $T>M_r$ give $\bar a(s,z)<T$; if $r\le s\le r+\varepsilon$, the same conclusion follows from the choice of $\varepsilon_2$.  Thus every $(s,z)\in(-\infty,r+\varepsilon]\times Y_1$ has preimage source height strictly less than $T$, which proves \eqref{eq:lower-contained}.

The two boundary hypersurfaces $\{r\}\times Y_1$ and $\Sigma_T$ are disjoint by \eqref{eq:upper-above}, and \eqref{eq:lower-contained} shows that the standard slice has a collar inside the first set in \eqref{eq:W10}.  Thus the boundary of $W_{10}$ is exactly
\[
  (\{r\}\times Y_1)\sqcup\Sigma_T.
\]
To prove compactness, use the negative-end statement of \eqref{eq:uniform-divergence} for $\Phi$: there is $S<T$ such that $a(t,y)<r$ whenever $t\le S$.  Therefore
\[
  \Phi^{-1}(W_{10})\subset[S,T]\times Y_0,
\]
which is compact.  The invertibility assertion follows from \cref{prop:smooth-cut-input}, whose two hypotheses hold by \eqref{eq:upper-above} and \eqref{eq:lower-contained}.
\end{proof}

\subsection{The Liouville structure}

We now equip the compact smooth cobordism constructed in
\cref{lem:choice-levels} with a Liouville structure.  Recall that
\[
W_{10}
=
\Phi((-\infty,T]\times Y_0)\cap([r,\infty)\times Y_1)
\]
has boundary
\[
\partial W_{10}
=
(\{r\}\times Y_1)\sqcup\Sigma_T,
\qquad
\Sigma_T=\Phi(\{T\}\times Y_0).
\]
The hypersurface $\Sigma_T$ is the outer boundary of the cut, inherited
from the boundary of $\Phi((-\infty,T]\times Y_0)$, whereas
$\{r\}\times Y_1$ is the lower boundary introduced by cutting with the
half-symplectization $[r,\infty)\times Y_1$.

\begin{proposition}[Forward Liouville $h$-cobordism]
\label{prop:forward}
For $r$ sufficiently large and $T$ chosen as in
\cref{lem:choice-levels}, the compact cobordism $W_{10}$ admits a
Liouville form $\Lambda_{10}$ for which
\[
(W_{10},\Lambda_{10})
\colon
(Y_1,\xi_1)\rightsquigarrow(Y_0,\xi_0)
\]
is a Liouville cobordism.  More precisely, the Liouville vector field
points inward along the lower boundary $\{r\}\times Y_1$ and outward
along the upper boundary $\Sigma_T$.  With the boundary
parametrizations specified below, these two boundary components carry
$\xi_1$ and $\xi_0$, respectively.

Moreover, the underlying smooth cobordism is invertible and hence is
an $h$-cobordism.
\end{proposition}

\begin{proof}
Choose $r$ satisfying \cref{lem:high-slice}, and then choose $T$ and
$W_{10}$ as in \cref{lem:choice-levels}.  Let
\[
j_T\colon Y_0\longrightarrow\Sigma_T,
\qquad
j_T(y)=\Phi(T,y),
\]
be the upper boundary parametrization.

Choose a smooth function
\[
\chi\colon W_{10}\longrightarrow[0,1]
\]
that is identically zero near the lower boundary and identically one
near the upper boundary.  Define
\begin{equation}\label{eq:Lambda10}
\Lambda_{10}
=
\lambda_1+\pi^*\beta+d(\chi G).
\end{equation}
Since $\beta$ is closed,
\begin{equation}\label{eq:dLambda}
d\Lambda_{10}
=
d\lambda_1
=
\omega_1|_{W_{10}}.
\end{equation}
Thus $\Lambda_{10}$ is a primitive of the fixed symplectic form on the
compact cut.  It remains to verify the contact structures and the
direction of the associated Liouville vector field on the two boundary
components.

\medskip
\noindent
\emph{The upper boundary.}
Near the upper boundary $\Sigma_T$ we have $\chi=1$, and hence
\eqref{eq:discrepancy-decomposition} gives
\[
\Lambda_{10}
=
\lambda_1+\pi^*\beta+dG
=
\mu.
\]
Consequently,
\begin{equation}\label{eq:upper-form}
j_T^*\Lambda_{10}
=
(\Phi\circ i_T)^*(\Phi^{-1})^*\lambda_0
=
i_T^*\lambda_0
=
e^T\alpha_0,
\end{equation}
where $i_T(y)=(T,y)$.  Thus the contact structure induced on
$\Sigma_T$, under the parametrization $j_T$, is precisely $\xi_0$.

We also check the direction of the Liouville vector field.  Since
\[
\mu=(\Phi^{-1})^*\lambda_0
\]
and $\partial_t$ is the Liouville vector field of $\lambda_0$ on
$\mathbb R\times Y_0$, the Liouville vector field of $\mu$ is
\[
\Phi_*\partial_t.
\]
Indeed,
\[
\iota_{\Phi_*\partial_t}d\mu
=
(\Phi^{-1})^*
\bigl(\iota_{\partial_t}d\lambda_0\bigr)
=
(\Phi^{-1})^*\lambda_0
=
\mu.
\]
Along $\Sigma_T=\Phi(\{T\}\times Y_0)$, the vector field
$\Phi_*\partial_t$ points out of
\[
\Phi((-\infty,T]\times Y_0).
\]
 Hence $\Sigma_T$ is the positive,
or convex, boundary of the Liouville cobordism.

\medskip
\noindent
\emph{The lower boundary.}
Near the lower boundary $\chi=0$, and hence
\begin{equation}\label{eq:lower-form}
\Lambda_{10}
=
e^t\alpha_1+\beta,
\end{equation}
where we suppress the pullback symbol for the one-form $\beta$ on
$Y_1$.  Its restriction to $\{r\}\times Y_1$ is
\[
\theta_{r,1}
=
e^r\alpha_1+\beta.
\]
By \cref{lem:high-slice}, the path \eqref{eq:theta} joins
$e^r\alpha_1$ to $\theta_{r,1}$ through positive contact forms.
Gray stability therefore gives a diffeomorphism
\[
\psi\in\operatorname{Diff}_0(Y_1)
\]
and a positive function $h\colon Y_1\to\mathbb R_{>0}$ such that
\begin{equation}\label{eq:gray-lower-form}
\psi^*\theta_{r,1}
=
he^r\alpha_1.
\end{equation}
We use
\[
j_r\colon Y_1\longrightarrow\{r\}\times Y_1,
\qquad
j_r(y)=(r,\psi(y)),
\]
as the lower boundary parametrization.  Then
\[
j_r^*\Lambda_{10}
=
\psi^*\theta_{r,1}
=
he^r\alpha_1,
\]
which is a positive contact form defining $\xi_1$.

It remains to prove that the Liouville vector field points inward along
this boundary.  Let $R_{\alpha_1}$ denote the Reeb vector field of
$\alpha_1$.  Using the splitting
\[
T(\mathbb R\times Y_1)
=
\mathbb R\partial_t
\oplus
\mathbb R R_{\alpha_1}
\oplus
\xi_1,
\]
write the Liouville vector field near the lower boundary as
\[
Z
=
a\partial_t+bR_{\alpha_1}+v,
\qquad
v\in\xi_1.
\]
By definition,
\[
\iota_Zd\Lambda_{10}
=
\Lambda_{10}.
\]
Since $d\beta=0$, equation \eqref{eq:lower-form} gives
\[
d\Lambda_{10}
=
d(e^t\alpha_1)
=
e^t(dt\wedge\alpha_1+d\alpha_1),
\]
and therefore
\begin{align*}
\iota_Zd\Lambda_{10}
&=
e^t\Bigl(
\iota_Z(dt\wedge\alpha_1)
+
\iota_Zd\alpha_1
\Bigr)\\
&=
e^t\Bigl(
dt(Z)\alpha_1-\alpha_1(Z)\,dt
+
\iota_vd\alpha_1
\Bigr)\\
&=
e^t\Bigl(
a\alpha_1-b\,dt+\iota_vd\alpha_1
\Bigr).
\end{align*}
Here we used
\[
dt(Z)=a,
\qquad
\alpha_1(Z)=b,
\qquad
\iota_{R_{\alpha_1}}d\alpha_1=0,
\]
together with the fact that $d\alpha_1$ has no $\partial_t$-component.

On the other hand, with respect to the decomposition
\[
TY_1
=
\mathbb R R_{\alpha_1}\oplus\xi_1,
\]
write
\[
\beta
=
\beta(R_{\alpha_1})\alpha_1+\beta_\xi,
\]
where
\[
\beta_\xi(R_{\alpha_1})=0,
\qquad
\beta_\xi|_{\xi_1}
=
\beta|_{\xi_1}.
\]
The Liouville equation
\[
\iota_Zd\Lambda_{10}
=
\Lambda_{10}
=
e^t\alpha_1+\beta
\]
therefore becomes
\[
e^t
\bigl(
a\alpha_1-b\,dt+\iota_vd\alpha_1
\bigr)
=
\bigl(e^t+\beta(R_{\alpha_1})\bigr)\alpha_1
+
\beta_\xi.
\]

We now compare the three components of this identity.  The
$dt$-component gives
\[
b=0.
\]
Evaluating on $R_{\alpha_1}$ gives
\[
e^ta
=
e^t+\beta(R_{\alpha_1}),
\]
and hence
\[
a
=
1+e^{-t}\beta(R_{\alpha_1}).
\]
Finally, restricting the identity to $\xi_1$ gives
\[
e^t\iota_vd\alpha_1
=
\beta|_{\xi_1},
\]
or equivalently
\[
\iota_vd\alpha_1
=
e^{-t}\beta|_{\xi_1}.
\]
Thus
\begin{equation}\label{eq:liouville-vector}
b=0,
\qquad
a=1+e^{-t}\beta(R_{\alpha_1}),
\qquad
\iota_vd\alpha_1=e^{-t}\beta|_{\xi_1}.
\end{equation}
The last equation determines $v\in\xi_1$ uniquely because
$d\alpha_1|_{\xi_1}$ is nondegenerate.

In particular,
\[
dt(Z)
=
a
=
1+e^{-t}\beta(R_{\alpha_1}).
\]
By the choice of $r$ in \cref{lem:high-slice},
\eqref{eq:positive-radial-component} gives
\[
dt(Z)>0
\qquad
\text{along }\{r\}\times Y_1.
\]
Now $W_{10}$ lies locally on the side $t\ge r$ of its lower boundary.
Hence the positive $t$-direction points into $W_{10}$ along
$\{r\}\times Y_1$.  It follows that $Z$ points inward there.
Therefore $\{r\}\times Y_1$ is the negative, or concave, boundary of
the Liouville cobordism.

We have shown that $(W_{10},\Lambda_{10})$ is a Liouville cobordism
\[
(Y_1,\xi_1)\rightsquigarrow(Y_0,\xi_0).
\]
Finally, the underlying marked smooth cobordism is invertible by
\cref{lem:choice-levels}.  Replacing the original lower boundary
marking by its composition with the Gray diffeomorphism $\psi$ does
not affect smooth invertibility.  Hence the underlying smooth
cobordism is invertible and, in particular, is an $h$-cobordism.
\end{proof}

\begin{proof}[Proof of \cref{thm:bidirectional-intro}]
Applying \cref{prop:forward} to $\Phi$ gives a Liouville
$h$-cobordism
\[
(Y_1,\xi_1)\rightsquigarrow(Y_0,\xi_0).
\]
Applying the same proposition to the inverse symplectomorphism
\[
\Phi^{-1}\colon S(Y_1,\xi_1)\longrightarrow S(Y_0,\xi_0)
\]
gives a Liouville $h$-cobordism in the opposite direction,
\[
(Y_0,\xi_0)\rightsquigarrow(Y_1,\xi_1).
\]
This proves the theorem.
\end{proof}
\begin{remark}[No symplectic cancellation is asserted]\label{rem:no-cancellation}
The two Liouville cobordisms need not be inverse morphisms in a symplectic cobordism category.  The argument proves smooth invertibility of each underlying cobordism, not that either Liouville structure is a product or that the two Liouville cobordisms cancel.
\end{remark}

\section{Floer contact invariants}\label{sec:floer}

Throughout this section, coefficients are in $\F=\Z/2\Z$.

\subsection{Monopole Floer naturality}

We use monopole Floer homology at this stage for a specific reason.
The cobordisms produced by
\cref{thm:bidirectional-intro} are Liouville, and hence strong
symplectic, cobordisms, but they are not known to be Stein or
Weinstein.

For the Ozsv\'ath--Szab\'o contact invariant, Hedden and
Tovstopyat-Nelip proved strong functoriality under Stein cobordisms
\cite[Theorem~1.3]{HeddenTovstopyatNelip2022}.  More precisely, if
$(W,J)$ is a Stein cobordism from $(Y_-,\xi_-)$ to
$(Y_+,\xi_+)$ and $\mathfrak{k}$ is the canonical
$\SpinC$ structure induced by $J$, then the Heegaard Floer
cobordism map associated to the reversed cobordism satisfies
\[
F_{W^\dagger,\mathfrak{k}}
\bigl(\widehat c_{\mathrm{HF}}(\xi_+)\bigr)
=
\widehat c_{\mathrm{HF}}(\xi_-),
\]
while the corresponding maps in the other $\SpinC$ structures
annihilate the contact class.

The analogous statement for arbitrary strong symplectic cobordisms
is not presently available in Heegaard Floer theory.  Indeed,
immediately after their Stein-cobordism naturality theorem,
Hedden and Tovstopyat-Nelip explicitly ask whether the contact
element is natural under strong symplectic cobordisms
\cite[Question~1.4]{HeddenTovstopyatNelip2022}.
To the best of our knowledge, this question remains open in the
generality needed here.  Mark and Tosun likewise point out that the
known equivalence between monopole Floer and Heegaard Floer homology
is not known to extend to the maps induced by general
four-dimensional cobordisms; in particular, one cannot simply
transport strong-cobordism naturality from monopole Floer homology
to Heegaard Floer homology
\cite{MarkTosun2022}.

In contrast, Echeverria proved precisely the required naturality
statement for the monopole Floer contact invariant under arbitrary
strong symplectic cobordisms \cite{Echeverria2020}.  Thus monopole
Floer homology provides the natural bridge between the Liouville
cobordisms constructed in \cref{thm:bidirectional-intro} and the
Heegaard Floer contact invariants used later in the classification
arguments.

Let
\[
  \widecheck c_{\mathrm{HM}}(\xi;\F)
  \in\widecheck{HM}_{\bullet}(-Y,\mathfrak s_\xi;\F)
\]
denote the monopole Floer contact invariant.  Echeverria's theorem
states that if $(W,\omega)$ is a strong symplectic cobordism from
$(Y_-,\xi_-)$ to $(Y_+,\xi_+)$, then the monopole Floer map
associated to the reversed cobordism, in the canonical symplectic
$\SpinC$ structure, sends
\begin{equation}\label{eq:echeverria}
  \widecheck c_{\mathrm{HM}}(\xi_+;\F)
  \longmapsto
  \widecheck c_{\mathrm{HM}}(\xi_-;\F).
\end{equation}
A Liouville cobordism is, in particular, a strong symplectic
cobordism.  We shall also use the following two consequences of
Echeverria's theorem \cite{Echeverria2020}:
\begin{equation}\label{eq:HM-fillability-overtwisted}
\begin{aligned}
  (Y,\xi)\text{ strongly symplectically fillable}
  &\quad\Longrightarrow\quad
  \widecheck c_{\mathrm{HM}}(\xi;\F)\ne0,\\
  \xi\text{ overtwisted}
  &\quad\Longrightarrow\quad
  \widecheck c_{\mathrm{HM}}(\xi;\F)=0.
\end{aligned}
\end{equation}

\begin{proof}[Proof of \cref{cor:HM-intro}]
Let $W_{01}$ be the Liouville cobordism from $(Y_0,\xi_0)$ to $(Y_1,\xi_1)$ supplied by \cref{thm:bidirectional-intro}.  By \eqref{eq:echeverria}, vanishing of $\widecheck c_{\mathrm{HM}}(\xi_1;\F)$ implies vanishing of $\widecheck c_{\mathrm{HM}}(\xi_0;\F)$.  The reverse Liouville cobordism $W_{10}$ gives the converse implication.
\end{proof}

\subsection{Comparison with the Heegaard Floer contact invariant}

A rational homology three-sphere $Y$ is an \emph{$L$-space} if
\[
  \dim_{\F}\widehat{HF}(Y;\F)=|H_1(Y;\Z)|.
\]
Equivalently, for every $\mathfrak s\in\SpinC(Y)$,
\[
  HF^+(Y,\mathfrak s;\F)\cong\Tplus
  :=\F[U,U^{-1}]/U\F[U]
\]
up to an absolute grading shift.

For a positive contact structure $\xi$ on $Y$, write
\[
  \widehat c_{\mathrm{HF}}(\xi;\F)
  \in\widehat{HF}(-Y,\mathfrak s_\xi;\F)
\]
for the Ozsv\'ath--Szab\'o contact invariant and
\[
  c^+_{\mathrm{HF}}(\xi;\F)
  \in HF^+(-Y,\mathfrak s_\xi;\F)
\]
for its plus version \cite{OzsvathSzabo2005}.  The natural map
\[
  \iota\colon\widehat{HF}(-Y,\mathfrak s_\xi;\F)
  \longrightarrow HF^+(-Y,\mathfrak s_\xi;\F)
\]
sends $\widehat c_{\mathrm{HF}}(\xi;\F)$ to $c^+_{\mathrm{HF}}(\xi;\F)$.

\begin{lemma}[Plus Heegaard Floer--monopole comparison]
\label{lem:plus-comparison}
For every closed positive cooriented contact three-manifold $(Y,\xi)$,
\begin{equation}\label{eq:plus-comparison}
  c^+_{\mathrm{HF}}(\xi;\F)=0
  \quad\Longleftrightarrow\quad
  \widecheck c_{\mathrm{HM}}(\xi;\F)=0.
\end{equation}
\end{lemma}

\begin{proof}
Choose a positive contact form $\alpha$ for $\xi$.
Colin--Ghiggini--Honda construct a $U$-equivariant isomorphism
\[
  \Phi^+\colon
  HF^+(-Y;\F)
  \xrightarrow{\;\cong\;}
  ECH(Y,\alpha;\F)
\]
and prove that, on homology, $\Phi^+$ sends the Ozsv\'ath--Szab\'o
plus contact class to the ECH contact class
\cite[Theorem~1.0.1]{ColinGhigginiHondaIII}.
Their theorem is stated with $\F=\mathbb Z/2\mathbb Z$
coefficients.

On the other hand, Taubes's isomorphism between embedded contact
homology and Seiberg--Witten Floer cohomology identifies the ECH
contact element with the monopole Floer contact element
\cite[Theorem~1.1]{TaubesECHSWV}.
Under the standard identification of the relevant Seiberg--Witten
Floer cohomology group with
$\widecheck{HM}_{\bullet}(-Y,\mathfrak s_\xi;\F)$, we therefore obtain
an isomorphism
\[
  HF^+(-Y,\mathfrak s_\xi;\F)
  \xrightarrow{\;\cong\;}
  \widecheck{HM}_{\bullet}(-Y,\mathfrak s_\xi;\F)
\]
which carries
$c^+_{\mathrm{HF}}(\xi;\F)$ to
$\widecheck c_{\mathrm{HM}}(\xi;\F)$.
Hence one contact class vanishes if and only if the other does.
\end{proof}

\begin{corollary}[Plus contact invariant]\label{cor:plus-vanishing}
If $(Y_0,\xi_0)$ and $(Y_1,\xi_1)$ have symplectomorphic symplectizations, then
\[
  c^+_{\mathrm{HF}}(\xi_0;\F)=0
  \quad\Longleftrightarrow\quad
  c^+_{\mathrm{HF}}(\xi_1;\F)=0.
\]
\end{corollary}

\begin{proof}
Combine \cref{cor:HM-intro,lem:plus-comparison}.
\end{proof}

\begin{lemma}[Contact-class comparison on $L$-spaces]
\label{lem:comparison}
Let $(Y,\xi)$ be a closed positive cooriented contact three-manifold
whose underlying oriented manifold is an $L$-space. Then
\begin{equation}\label{eq:comparison}
  \widehat c_{\mathrm{HF}}(\xi;\mathbb F_2)=0
  \quad\Longleftrightarrow\quad
  \widecheck c_{\mathrm{HM}}(\xi;\mathbb F_2)=0.
\end{equation}
\end{lemma}

\begin{proof}
Fix
\[
  \mathfrak s=\mathfrak s_\xi.
\]
Recall that the natural map
\[
  \iota\colon
  \widehat{HF}(-Y,\mathfrak s;\mathbb F_2)
  \longrightarrow
  HF^+(-Y,\mathfrak s;\mathbb F_2)
\]
appearing in the standard exact sequence relating the hat and plus
versions satisfies
\begin{equation}\label{eq:hat-plus-kernel}
  \operatorname{im}\iota
  =
  \ker\!\left(
    U\colon
    HF^+(-Y,\mathfrak s;\mathbb F_2)
    \longrightarrow
    HF^+(-Y,\mathfrak s;\mathbb F_2)
  \right).
\end{equation}
Moreover, the Ozsv\'ath--Szab\'o contact invariants in the two versions
are related by
\begin{equation}\label{eq:hat-to-plus-contact}
  \iota\bigl(
    \widehat c_{\mathrm{HF}}(\xi;\mathbb F_2)
  \bigr)
  =
  c^+_{\mathrm{HF}}(\xi;\mathbb F_2).
\end{equation}
In particular,
\[
  U\,c^+_{\mathrm{HF}}(\xi;\mathbb F_2)=0
\]
by \eqref{eq:hat-plus-kernel} and
\eqref{eq:hat-to-plus-contact}.

Because $Y$ is an $L$-space, for every
$\mathfrak s\in\SpinC(Y)$ one has
\[
  HF^+(-Y,\mathfrak s;\mathbb F_2)
  \cong
  \mathcal T^+
\]
up to an absolute grading shift, where
\[
  \mathcal T^+
  =
  \mathbb F_2[U,U^{-1}]
  /
  U\mathbb F_2[U].
\]
Hence
\[
  \dim_{\mathbb F_2}
  \ker\!\left(
    U\colon
    HF^+(-Y,\mathfrak s;\mathbb F_2)
    \longrightarrow
    HF^+(-Y,\mathfrak s;\mathbb F_2)
  \right)
  =1.
\]
The $L$-space condition also gives
\[
  \dim_{\mathbb F_2}
  \widehat{HF}(-Y,\mathfrak s;\mathbb F_2)
  =1.
\]
Therefore \eqref{eq:hat-plus-kernel} implies that
\[
  \iota\colon
  \widehat{HF}(-Y,\mathfrak s;\mathbb F_2)
  \xrightarrow{\;\cong\;}
  \ker U
  \subset
  HF^+(-Y,\mathfrak s;\mathbb F_2)
\]
is an isomorphism. Together with
\eqref{eq:hat-to-plus-contact}, this gives
\begin{equation}\label{eq:hat-plus-vanishing}
  \widehat c_{\mathrm{HF}}(\xi;\mathbb F_2)=0
  \quad\Longleftrightarrow\quad
  c^+_{\mathrm{HF}}(\xi;\mathbb F_2)=0.
\end{equation}

On the other hand, \cref{lem:plus-comparison} gives
\begin{equation}\label{eq:plus-monopole-vanishing}
  c^+_{\mathrm{HF}}(\xi;\mathbb F_2)=0
  \quad\Longleftrightarrow\quad
  \widecheck c_{\mathrm{HM}}(\xi;\mathbb F_2)=0.
\end{equation}
Combining
\eqref{eq:hat-plus-vanishing} and
\eqref{eq:plus-monopole-vanishing}
proves \eqref{eq:comparison}.
\end{proof}
\begin{corollary}[Heegaard Floer vanishing on $L$-spaces]\label{cor:HF-vanishing}
Suppose that $(Y_0,\xi_0)$ and $(Y_1,\xi_1)$ have symplectomorphic symplectizations and that the underlying manifolds are $L$-spaces.  Then
\[
  \widehat c_{\mathrm{HF}}(\xi_0;\F)=0
  \quad\Longleftrightarrow\quad
  \widehat c_{\mathrm{HF}}(\xi_1;\F)=0.
\]
\end{corollary}

\begin{proof}
Combine \cref{cor:HM-intro,lem:comparison}.
\end{proof}

\section{Classwise rigidity criteria and connected sums}\label{sec:criterion}

The formal theorem identifies the full oriented plane-field homotopy class, while the bidirectional Liouville cobordisms preserve vanishing of the monopole contact invariant.  This gives a criterion without any $L$-space hypothesis.

\begin{proposition}[Monopole classwise criterion]\label{prop:HM-criterion}
Let $\mathscr D$ be a class of closed connected oriented three-manifolds.  Assume that, for every $Y\in\mathscr D$:
\begin{enumerate}[label=\textup{(\roman*)},leftmargin=2.7em]
  \item a positive cooriented contact structure $\xi$ on $Y$ is tight if and only if $\widecheck c_{\mathrm{HM}}(\xi;\F)\ne0$;
  \item two tight positive cooriented contact structures on $Y$ that are homotopic as oriented plane fields are contactomorphic.
\end{enumerate}
Then the symplectization question has an affirmative answer for arbitrary positive cooriented contact structures on manifolds in $\mathscr D$.
\end{proposition}

\begin{proof}
A contactomorphism lifts to a strict exact symplectomorphism by \cref{lem:strict-lift}.  Conversely, suppose that
\[
  S(Y_0,\xi_0)\congsymp S(Y_1,\xi_1),
  \qquad Y_0,Y_1\in\mathscr D.
\]
By \cref{thm:formal-rigidity}, there is an orientation-preserving diffeomorphism $f\colon Y_0\to Y_1$ such that $\xi_0$ and $f^*\xi_1$ are homotopic as oriented plane fields.  Replacing $(Y_1,\xi_1)$ by $(Y_0,f^*\xi_1)$, we may assume that both structures lie on the same manifold and are homotopic as oriented plane fields.

If both structures are overtwisted, Eliashberg's classification \cite{Eliashberg1989} and Gray stability give a contactomorphism.  If exactly one is tight, assumption~\textup{(i)} says that one monopole contact invariant is nonzero and the other is zero, contradicting \cref{cor:HM-intro}.  If both are tight, assumption~\textup{(ii)} gives a contactomorphism.
\end{proof}

As introduced in the introduction, for a class $\mathscr D$ of closed connected oriented three-manifolds, we write $\mathscr D^{\#}$ for the class of all nonempty finite connected sums of members of $\mathscr D$.

\begin{lemma}[Connected-sum formulas]\label{lem:connected-sum-formulas}

Let $Y_1,\ldots,Y_k$ be closed connected oriented three-manifolds, and set
\[
  Y
  =
  Y_1\mathbin{\#}\cdots\mathbin{\#}Y_k.
\]
Then there is a canonical connected-sum bijection
\begin{equation}\label{eq:spinc-connected-sum}
  \SpinC(Y_1)\times\cdots\times\SpinC(Y_k)
  \longrightarrow
  \SpinC(Y),
  \qquad
  (\mathfrak s_1,\ldots,\mathfrak s_k)
  \longmapsto
  \mathfrak s_1\mathbin{\#}\cdots\mathbin{\#}\mathfrak s_k.
\end{equation}
For $\mathfrak s_j\in\SpinC(Y_j)$, the Heegaard Floer K\"unneth theorem gives
\begin{equation}\label{eq:HF-connected-sum}
  \widehat{HF}\bigl(Y,
    \mathfrak s_1\mathbin{\#}\cdots\mathbin{\#}\mathfrak s_k;\F\bigr)
  \cong
  \bigotimes_{j=1}^k\widehat{HF}(Y_j,\mathfrak s_j;\F).
\end{equation}
In particular, a finite connected sum of $L$-spaces is an $L$-space.

If $\xi_j$ is a positive cooriented contact structure on $Y_j$, then
\begin{equation}\label{eq:contact-spinc-connected-sum}
  \mathfrak s_{\xi_1\#\cdots\#\xi_k}
  =
  \mathfrak s_{\xi_1}\mathbin{\#}\cdots\mathbin{\#}\mathfrak s_{\xi_k},
\end{equation}
and, under the K\"unneth isomorphism for the orientation-reversed manifolds,
\begin{equation}\label{eq:contact-class-connected-sum}
  \widehat c_{\mathrm{HF}}(\xi_1\#\cdots\#\xi_k;\F)
  =
  \bigotimes_{j=1}^k
  \widehat c_{\mathrm{HF}}(\xi_j;\F).
\end{equation}
\end{lemma}

\begin{proof}
The bijection \eqref{eq:spinc-connected-sum} is obtained by restricting $\SpinC$ structures to the punctured summands and gluing across the separating spheres; restriction is its inverse.  Formula \eqref{eq:HF-connected-sum} is the hat version of the K\"unneth theorem \cite[Theorem~1.5]{OzsvathSzabo2004Properties}.  If every $Y_j$ is an $L$-space, each factor on the right is one-dimensional in every $\SpinC$ structure, while
\[
  |H_1(Y;\Z)|=\prod_{j=1}^k|H_1(Y_j;\Z)|,
\]
which proves the $L$-space assertion.

The contact connected sum agrees with the original contact plane fields outside the chosen Darboux balls and the standard neck, which gives \eqref{eq:contact-spinc-connected-sum}.  To obtain \eqref{eq:contact-class-connected-sum}, choose supporting open books with connected bindings.  The boundary connected sum of their pages supports the contact connected sum, and the connected-sum identification of the filtered knot Floer complexes carries the distinguished contact generator to the tensor product of the distinguished generators \cite[Proposition~2.1]{OzsvathSzabo2005}.  Iteration proves the formula for arbitrary $k$.  Working over $\F=\Z/2\Z$ removes the sign ambiguity.
\end{proof}

\begin{lemma}[Colin's decomposition for a fixed connected sum]\label{lem:colin-decomposition}
Let $Y_1,\ldots,Y_k$ be closed connected oriented three-manifolds, and fix a connected-sum decomposition
\[
  Y=Y_1\mathbin{\#}\cdots\mathbin{\#}Y_k.
\]
If $\xi$ is a tight contact structure on $Y$, then there are tight contact structures $\xi_j$ on $Y_j$, unique up to contact isotopy, such that
\[
  \xi\simeq\xi_1\#\cdots\#\xi_k
\]
through positive contact structures.  Conversely, the contact connected sum of tight contact structures is tight.
\end{lemma}

\begin{proof}
This is Colin's connected-sum theorem \cite{Colin1997}; see also the
formulation for a fixed connected-sum decomposition in
\cite[Sections~1--2]{DingGeiges2007}.
\end{proof}

\begin{lemma}[Adding a strongly fillable summand]
\label{lem:add-fillable-summand}
Let $(A,\zeta)$ and $(B,\eta)$ be closed connected positive cooriented contact three-manifolds.  If $(B,\eta)$ is strongly symplectically fillable, then there is a strong symplectic cobordism
\[
  (A,\zeta)\rightsquigarrow
  (A\mathbin{\#}B,\zeta\#\eta).
\]
\end{lemma}

\begin{proof}
Let $(X_B,\omega_B)$ be a strong filling of $(B,\eta)$, and take a compact piece of the symplectization of $(A,\zeta)$ as a trivial strong cobordism from $(A,\zeta)$ to itself.  Regard the disjoint union of these two symplectic manifolds as a strong cobordism whose negative boundary is $(A,\zeta)$ and whose positive boundary is
\[
  (A,\zeta)\sqcup(B,\eta).
\]
Attach a Weinstein one-handle along Darboux balls, one in each component of the positive boundary.  The negative boundary is unchanged, while the positive boundary becomes the contact connected sum
\[
  (A\mathbin{\#}B,\zeta\#\eta).
\]
The resulting cobordism is strong symplectic.
\end{proof}

\begin{proposition}[Mixed connected-sum closure of the rigidity criterion]
\label{prop:connected-sum-closure}
Let $\mathscr A$ be a class of closed connected oriented $L$-spaces, and let $\mathscr V$ be a class of closed connected oriented three-manifolds.  Assume that the following conditions hold.

For every $Y\in\mathscr A$:
\begin{enumerate}[label=\textup{(A\arabic*)},leftmargin=2.9em]
  \item a positive cooriented contact structure on $Y$ is tight if and only if its hat Heegaard Floer contact invariant is nonzero;
  \item two tight positive cooriented contact structures on $Y$ with isomorphic induced $\SpinC$ structures are contact isotopic.
\end{enumerate}
For every $Y\in\mathscr V$:
\begin{enumerate}[label=\textup{(V\arabic*)},leftmargin=2.9em]
  \item every tight positive cooriented contact structure on $Y$ is strongly symplectically fillable;
  \item any two tight positive cooriented contact structures on $Y$ are contactomorphic.
\end{enumerate}

Then, for every $Y\in(\mathscr A\cup\mathscr V)^{\#}$ and every positive cooriented contact structure $\xi$ on $Y$,
\begin{equation}\label{eq:mixed-tight-HM}
  \xi\text{ is tight}
  \quad\Longleftrightarrow\quad
  \widecheck c_{\mathrm{HM}}(\xi;\F)\ne0.
\end{equation}
Moreover, two tight positive cooriented contact structures on $Y$ that are homotopic as oriented plane fields are contactomorphic.  Consequently, the symplectization question has an affirmative answer for arbitrary positive cooriented contact structures on manifolds in $(\mathscr A\cup\mathscr V)^{\#}$.
\end{proposition}

\begin{proof}
Fix a connected-sum presentation
\begin{equation}\label{eq:mixed-presentation}
  Y
  =A_1\mathbin{\#}\cdots\mathbin{\#}A_p
   \mathbin{\#}V_1\mathbin{\#}\cdots\mathbin{\#}V_q,
  \qquad
  A_i\in\mathscr A,
  \quad
  V_j\in\mathscr V,
\end{equation}
where $p,q\ge0$ and $p+q\ge1$.  Empty strings of factors are omitted.

We first prove \eqref{eq:mixed-tight-HM}.  Suppose that $\xi$ is tight.  By \cref{lem:colin-decomposition}, there are tight contact structures $\xi_i$ on $A_i$ and $\eta_j$ on $V_j$ such that
\begin{equation}\label{eq:mixed-contact-decomposition}
  \xi\simeq
  \xi_1\#\cdots\#\xi_p\#
  \eta_1\#\cdots\#\eta_q
\end{equation}
through positive contact structures.

Assume first that $p>0$, and set
\[
  A=A_1\mathbin{\#}\cdots\mathbin{\#}A_p,
  \qquad
  \xi_A=\xi_1\#\cdots\#\xi_p.
\]
By assumption~\textup{(A1)}, every $\widehat c_{\mathrm{HF}}(\xi_i;\F)$ is nonzero.  The connected-sum formula \eqref{eq:contact-class-connected-sum} therefore gives
\[
  \widehat c_{\mathrm{HF}}(\xi_A;\F)\ne0.
\]
The manifold $A$ is an $L$-space by \cref{lem:connected-sum-formulas}, so \cref{lem:comparison} implies
\begin{equation}\label{eq:A-part-HM-nonzero}
  \widecheck c_{\mathrm{HM}}(\xi_A;\F)\ne0.
\end{equation}

If $q=0$, equations \eqref{eq:mixed-contact-decomposition} and \eqref{eq:A-part-HM-nonzero} already prove that $\widecheck c_{\mathrm{HM}}(\xi;\F)\ne0$.  Suppose therefore that $q>0$, and put
\[
  B=V_1\mathbin{\#}\cdots\mathbin{\#}V_q,
  \qquad
  \eta_B=\eta_1\#\cdots\#\eta_q.
\]
Each $(V_j,\eta_j)$ is strongly fillable by assumption~\textup{(V1)}.  Taking the disjoint union of strong fillings and attaching Weinstein one-handles shows that $(B,\eta_B)$ is strongly fillable.  By \cref{lem:add-fillable-summand}, there is a strong symplectic cobordism
\[
  W\colon(A,\xi_A)\rightsquigarrow
  (A\mathbin{\#}B,\xi_A\#\eta_B).
\]
Echeverria's naturality gives
\[
  \widecheck{HM}_{\bullet}(W^\dagger,\mathfrak s_\omega)
  \bigl(\widecheck c_{\mathrm{HM}}(\xi_A\#\eta_B;\F)\bigr)
  =\widecheck c_{\mathrm{HM}}(\xi_A;\F).
\]
The right-hand side is nonzero by \eqref{eq:A-part-HM-nonzero}.  Hence the contact invariant of $\xi_A\#\eta_B$ is nonzero.  Since \eqref{eq:mixed-contact-decomposition} identifies $\xi$ with this connected-sum structure up to contact isotopy, we obtain
\[
  \widecheck c_{\mathrm{HM}}(\xi;\F)\ne0.
\]

If $p=0$, all summands in \eqref{eq:mixed-presentation} belong to $\mathscr V$.  Assumption~\textup{(V1)} and the Weinstein one-handle construction show directly that $(Y,\xi)$ is strongly fillable.  Its monopole contact invariant is therefore nonzero by \eqref{eq:HM-fillability-overtwisted}.  We have now proved nonvanishing for every tight $\xi$.  Conversely, if $\xi$ is overtwisted, then \eqref{eq:HM-fillability-overtwisted} gives
\[
  \widecheck c_{\mathrm{HM}}(\xi;\F)=0.
\]
This proves \eqref{eq:mixed-tight-HM}.

It remains to prove the uniqueness assertion.  Let $\xi^0$ and $\xi^1$ be tight contact structures on $Y$ that are homotopic as oriented plane fields.  Applying \cref{lem:colin-decomposition} to the fixed presentation \eqref{eq:mixed-presentation}, write
\[
  \xi^a\simeq
  \xi_1^a\#\cdots\#\xi_p^a\#
  \eta_1^a\#\cdots\#\eta_q^a,
  \qquad a=0,1,
\]
where all summand contact structures are tight.  Plane-field homotopy implies
\[
  \mathfrak s_{\xi^0}\cong\mathfrak s_{\xi^1}.
\]
By \eqref{eq:contact-spinc-connected-sum} and the injectivity of the connected-sum bijection \eqref{eq:spinc-connected-sum},
\[
  \mathfrak s_{\xi_i^0}\cong\mathfrak s_{\xi_i^1}
  \quad(1\le i\le p),
  \qquad
  \mathfrak s_{\eta_j^0}\cong\mathfrak s_{\eta_j^1}
  \quad(1\le j\le q).
\]
Assumption~\textup{(A2)} gives contact isotopies $\xi_i^0\simeq\xi_i^1$ on the $\mathscr A$-factors, while assumption~\textup{(V2)} gives contactomorphisms
\[
  (V_j,\eta_j^0)\congcont(V_j,\eta_j^1)
\]
on the $\mathscr V$-factors.  By the well-definedness of contact connected sum, taking the connected sum of these factorwise contactomorphisms gives
\[
  (Y,\xi^0)\congcont(Y,\xi^1).
\]
Thus both hypotheses of \cref{prop:HM-criterion} hold for $(\mathscr A\cup\mathscr V)^{\#}$, and that proposition proves the final assertion.
\end{proof}

\section{Strong fillability and further classes}\label{sec:strong}

\subsection{Strong fillability as a symplectization invariant}

Recall that a \emph{strong symplectic filling} of $(Y,\xi)$ is a compact symplectic four-manifold $(X,\omega)$ with $\partial X=Y$ for which a Liouville vector field is defined near the boundary, points outward, and induces $\xi$.  We use the standard gluing of a strong filling to a strong cobordism after adjusting their Liouville collars; see, for example, \cite[Chapter~5]{Geiges2008}.

\begin{proof}[Proof of \cref{cor:strong-fillability-intro}]
Let $W_{01}$ and $W_{10}$ be the Liouville cobordisms supplied by \cref{thm:bidirectional-intro}.  If $(X_0,\omega_0)$ strongly fills $(Y_0,\xi_0)$, then gluing $X_0$ to $W_{01}$ along $(Y_0,\xi_0)$ gives a strong filling of $(Y_1,\xi_1)$.  The standard collar adjustment makes the boundary contact forms agree before gluing.  The reverse cobordism $W_{10}$ proves the converse.
\end{proof}

\begin{theorem}[Strong-fillability criterion and connected-sum closure]
\label{thm:strong-criterion}
Let $\mathscr D$ be a class of closed connected oriented three-manifolds such that, for every $Y\in\mathscr D$:
\begin{enumerate}[label=\textup{(\roman*)},leftmargin=2.7em]
  \item every tight positive cooriented contact structure on $Y$ is strongly symplectically fillable;
  \item two tight positive cooriented contact structures on $Y$ that are homotopic as oriented plane fields are contactomorphic.
\end{enumerate}
Then the symplectization question has an affirmative answer for arbitrary positive cooriented contact structures on manifolds in $\mathscr D$.

Moreover, suppose that assumption~\textup{(ii)} is strengthened to:
\begin{enumerate}[label=\textup{(\roman*${}'$)},leftmargin=2.7em,start=2]
  \item any two tight positive cooriented contact structures on $Y$ are contactomorphic.
\end{enumerate}
Then every tight positive cooriented contact structure on every manifold in $\mathscr D^{\#}$ is strongly symplectically fillable, any two such tight contact structures are contactomorphic, and the symplectization question has an affirmative answer for arbitrary positive cooriented contact structures on manifolds in $\mathscr D^{\#}$.
\end{theorem}

\begin{proof}[Proof of \cref{thm:strong-criterion}]
We first prove the general rigidity criterion.

Let $Y_0,Y_1\in\mathscr D$, and suppose that
\[
  S(Y_0,\xi_0)
  \cong_{\mathrm{symp}}
  S(Y_1,\xi_1).
\]
By \cref{thm:formal-rigidity}, there is an orientation-preserving
diffeomorphism
\[
  f\colon Y_0\longrightarrow Y_1
\]
such that
\[
  \xi_0\simeq f^*\xi_1
\]
as oriented plane fields.  Pulling $\xi_1$ back by $f$, it suffices to compare $\xi_0$ and $f^*\xi_1$ on the manifold $Y_0\in\mathscr D$.  We therefore assume that the two contact structures lie on the same manifold $Y\in\mathscr D$ and are homotopic as oriented plane fields.

If both contact structures are overtwisted, Eliashberg's
classification theorem, followed by Gray stability, gives a
contactomorphism.

Suppose next that exactly one of them is tight.  By
assumption~\textup{(i)}, the tight structure is strongly
symplectically fillable.  Since the two symplectizations are
symplectomorphic, \cref{cor:strong-fillability-intro} implies that the
other contact structure is also strongly symplectically fillable.
This is impossible, since strong symplectic fillability implies
tightness \cite[Chapter~5]{Geiges2008}.  Alternatively, on any class
for which tightness is equivalent to nonvanishing of
$\widecheck c_{\mathrm{HM}}$, the mixed case is excluded directly by
\cref{cor:HM-intro}.

Finally, if both contact structures are tight, their homotopy as
oriented plane fields and assumption~\textup{(ii)} imply that they are
contactomorphic.

Thus symplectomorphic symplectizations imply contactomorphic contact
manifolds for manifolds in $\mathscr D$.  The claim that a
contactomorphism induces a strict exact symplectomorphism of the
symplectizations follows from \cref{lem:strict-lift}.  This proves the
first assertion of the theorem.

We now prove the connected-sum assertion under the stronger
assumption~\textup{(ii$'$)}.  Let
\[
  Y
  =
  Y_1\mathbin{\#}\cdots\mathbin{\#}Y_k,
  \qquad
  Y_j\in\mathscr D.
\]

Let $\xi$ be a tight positive cooriented contact structure on $Y$.
By \cref{lem:colin-decomposition}, there are tight contact structures
$\xi_j$ on $Y_j$ such that
\[
  \xi
  \simeq
  \xi_1\#\cdots\#\xi_k
\]
through positive contact structures.  By assumption~\textup{(i)},
each $(Y_j,\xi_j)$ is strongly symplectically fillable.  Taking the
disjoint union of strong fillings of the $(Y_j,\xi_j)$ and attaching
Weinstein one-handles between their convex boundary components
produces a strong symplectic filling of the contact connected sum
\[
  (Y_1,\xi_1)\#\cdots\#(Y_k,\xi_k).
\]
Since $\xi$ is contact isotopic to this connected-sum structure, it is
also strongly symplectically fillable.  Hence every tight positive
cooriented contact structure on every manifold in
$\mathscr D^{\#}$ is strongly symplectically fillable.

It remains to prove uniqueness of the tight contactomorphism class.
Let $\xi^0$ and $\xi^1$ be two tight positive cooriented contact
structures on $Y$.  Applying \cref{lem:colin-decomposition} to the
fixed connected-sum presentation of $Y$, we obtain tight contact
structures $\xi_j^a$ on $Y_j$, for $a=0,1$, such that
\[
  \xi^a
  \simeq
  \xi_1^a\#\cdots\#\xi_k^a,
  \qquad
  a=0,1.
\]
By assumption~\textup{(ii$'$)}, for each $j$ there is a
contactomorphism
\[
  (Y_j,\xi_j^0)
  \cong_{\mathrm{cont}}
  (Y_j,\xi_j^1).
\]
Taking the contact connected sum of these contactomorphisms gives
\[
  (Y,\xi_1^0\#\cdots\#\xi_k^0)
  \cong_{\mathrm{cont}}
  (Y,\xi_1^1\#\cdots\#\xi_k^1).
\]
Therefore
\[
  (Y,\xi^0)
  \cong_{\mathrm{cont}}
  (Y,\xi^1).
\]
Thus any two tight positive cooriented contact structures on a
manifold in $\mathscr D^{\#}$ are contactomorphic.

Consequently, every manifold in $\mathscr D^{\#}$ satisfies
assumption~\textup{(i)} and, in fact, the stronger version
\textup{(ii$'$)} of assumption~\textup{(ii)}.  Applying the first
assertion of the theorem to the class $\mathscr D^{\#}$ proves
symplectization rigidity for arbitrary positive cooriented contact
structures on manifolds in $\mathscr D^{\#}$.
\end{proof}

\subsection{The manifold \texorpdfstring{$S^1\times S^2$}{S1 times S2}}

\begin{proposition}\label{prop:s1s2-property}
Every tight positive cooriented contact structure on $S^1\times S^2$ is contact isotopic to the standard one, and the standard structure is Stein fillable.
\end{proposition}

\begin{proof}
Uniqueness is classical; see \cite[Theorem~4.10.1]{Geiges2008}.  The standard structure is the contact boundary of the Weinstein one-handlebody $S^1\times D^3$.
\end{proof}

\subsection{Surgeries on the figure-eight knot}
Recall that
\[
  E(r):=S^3_r(4_1)
\]
for smooth $r$-surgery on the figure-eight knot, with the surgery orientation, and 
\[
\begin{split}
  \mathscr E={}&\{E(n):n>0,\ n\ne4\}
  \cup\{E(-1),E(-2),E(-3)\}\\
  &\cup\{E(-1/m):m>1\}.
\end{split}
\]

\begin{proposition}[Conway--Min]\label{prop:figure-eight-property}
For every $Y\in\mathscr E$, every tight positive cooriented contact structure on $Y$ is strongly symplectically fillable, and any two tight positive cooriented contact structures on $Y$ are contactomorphic.
\end{proposition}

\begin{proof}
Conway--Min prove that all tight structures in their classification range are strongly symplectically fillable \cite[Theorem~1.4]{ConwayMin2020}.  For an integer $n>0$, $n\ne4$, the manifold $E(n)$ has exactly two tight isotopy classes \cite[Corollary~1.3]{ConwayMin2020}.  Both are universally tight by \cite[Theorem~1.6]{ConwayMin2020}, and the $180^\circ$ rotational symmetry used in the proof of that theorem gives a contactomorphism between them.  For $n=-1,-2,-3$, there is a unique tight isotopy class by \cite[Corollary~1.3]{ConwayMin2020}.  Finally, for every $m>1$, the integer homology sphere $E(-1/m)$ has exactly two tight isotopy classes, both Stein fillable and contactomorphic \cite[Corollary~1.7]{ConwayMin2020}.
\end{proof}

\subsection{Brieskorn manifolds}
Recall that
\[
\begin{aligned}
\mathscr B={}&
 \{-\Sigma(2,3,5),\ \Sigma(2,3,5),\ -\Sigma(2,3,7),\ -\Sigma(2,3,11)\}\\
&\cup\{\Sigma(2,3,6k+1):k\ge1\}
 \cup\{\Sigma(2,3,6k-1):k\ge2\}\\
&\cup\{\Sigma(2,5,7),\Sigma(2,5,9),\Sigma(3,4,5),
        \Sigma(3,7,10),\Sigma(2,7,11),\Sigma(3,7,19),\\
&\hspace{3.7em}\Sigma(3,4,11),\Sigma(3,8,11),\Sigma(3,5,7),
        \Sigma(4,5,9),\Sigma(3,5,14)\}.
\end{aligned}
\]
\begin{proposition}[Cavallo]\label{prop:brieskorn-property}
For every $Y\in\mathscr B$, every tight positive cooriented contact structure on $Y$ is strongly symplectically fillable, and any two tight positive cooriented contact structures on $Y$ are contactomorphic.
\end{proposition}

\begin{proof}
Cavallo records that $-\Sigma(2,3,5)$ has no tight contact structure, while $\Sigma(2,3,5)$, $-\Sigma(2,3,7)$, and $-\Sigma(2,3,11)$ have a unique tight isotopy class; the latter classes are Stein fillable \cite[Introduction]{Cavallo2026}.  For each remaining canonically oriented Brieskorn sphere listed in \eqref{eq:brieskorn-class}, Cavallo proves that there are exactly two fillable tight isotopy classes and no other tight structures; the two are the canonical Milnor fillable structure and its conjugate, and they are contactomorphic \cite[Theorem~1.1 and Proposition~1.2]{Cavallo2026}.  The fillable structures are presented by Legendrian surgery in \cite[Section~2]{Cavallo2026}, hence are Stein fillable.
\end{proof}

\begin{remark}[Scope of the strong-fillability applications]\label{rem:strong-scope}
The figure-eight list \eqref{eq:figure-eight-class} is deliberately restricted to slopes for which Conway--Min's classification gives at most one tight contactomorphism class.  Other rational slopes in their classification range may carry several non-contactomorphic tight structures and require a finer plane-field calculation.  Likewise, Cavallo's four-fiber examples $\Sigma(2,3,7,41)$ and $\Sigma(2,3,11,13)$ are not included in $\mathscr B$: besides their two fillable structures, they carry nonfillable tight structures with Giroux torsion, so the first hypothesis of \cref{thm:strong-criterion} fails.

The unified class $\mathscr U$ is larger than the strongly fillable single-class part $\mathscr F$.  In particular, \cref{thm:main} does not claim that every tight contact structure on a manifold in $\mathscr U^{\#}$ is strongly fillable.  The $L$-space factors in $\mathscr S\cup\mathscr W$ are handled instead by contact-invariant nonvanishing, while strong fillability is used only for the $\mathscr F$-part in order to construct the cobordism of \cref{lem:add-fillable-summand}.
\end{remark}

\section{Small Seifert fibered \texorpdfstring{$L$}{L}-spaces}\label{sec:seifert}

We first use Matkovi\v{c}'s classification.  Her terminology ``small Seifert fibered space'' refers to a Seifert fibration over $S^2$ with exactly three exceptional fibers.  The final corollary of her classification also includes the cases with fewer than three exceptional fibers, which are lens spaces.

\begin{theorem}[Matkovi\v{c} \cite{Matkovic2018}]\label{thm:Matkovic}
Let $Y$ be an oriented $L$-space admitting a Seifert fibration over $S^2$ with at most three exceptional fibers.  For every positive cooriented contact structure $\xi$ on $Y$,
\[
  \xi\text{ is tight}
  \quad\Longleftrightarrow\quad
  \widehat c_{\mathrm{HF}}(\xi;\F)\ne0.
\]
Moreover, if $\xi$ and $\xi'$ are tight positive cooriented contact structures on $Y$, then
\[
  \xi\text{ and }\xi'\text{ are contact isotopic}
  \quad\Longleftrightarrow\quad
  \mathfrak s_\xi\cong\mathfrak s_{\xi'}.
\]
\end{theorem}

This is \cite[Corollary~1.4]{Matkovic2018}.

Thus every member of $\mathscr S$ satisfies conditions~\textup{(A1)} and~\textup{(A2)} of \cref{prop:connected-sum-closure}.

\section{Whitehead-link surgery \texorpdfstring{$L$}{L}-spaces}\label{sec:whitehead}

We retain the notation $M(n,r)$ and $\mathcal R_+$ from the introduction.

\begin{proposition}[Liu \cite{Liu2017}]\label{prop:Liu-whitehead}
For the Whitehead link, the rational surgery manifold $M(r_1,r_2)$ is an $L$-space if and only if $r_1>0$ and $r_2>0$.
\end{proposition}
\begin{proof}
This is \cite[Proposition~6.4]{Liu2017}.  In particular, $M(n,r)$ is an $L$-space whenever $n\ge5$ and $r\in\mathcal R_+$.
\end{proof}

\begin{theorem}[Min--Nonino \cite{MinNonino2026}]\label{thm:MinNonino}
Let $n\ge5$ be an integer and let $r\in\mathcal R_+$.  Then every tight positive cooriented contact structure on $M(n,r)$ is Stein fillable.  Moreover, distinct tight contact isotopy classes on $M(n,r)$ have distinct hat Heegaard Floer contact invariants.
\end{theorem}
\begin{proof}
This is the part of \cite[Theorems~1.2 and~1.6]{MinNonino2026} corresponding to $r\in\mathcal R_+$.
\end{proof}

\begin{lemma}[The rigidity hypotheses for $M(n,r)$]\label{lem:whitehead-criterion}
Let $n\ge5$ and $r\in\mathcal R_+$.  For positive cooriented contact structures on $Y=M(n,r)$:
\begin{enumerate}[label=\textup{(\roman*)},leftmargin=2.7em]
  \item $\xi$ is tight if and only if $\widehat c_{\mathrm{HF}}(\xi;\F)\ne0$;
  \item two tight structures with isomorphic induced $\SpinC$ structures are contact isotopic.
\end{enumerate}
\end{lemma}

\begin{proof}
The Ozsv\'ath--Szab\'o contact invariant vanishes for overtwisted structures and is nonzero for Stein fillable structures \cite{OzsvathSzabo2005}.  Since every tight structure on $Y$ is Stein fillable by \cref{thm:MinNonino}, this proves~\textup{(i)}.

For~\textup{(ii)}, let $\xi_0$ and $\xi_1$ be tight and suppose that their induced $\SpinC$ structures are isomorphic.  Write
\[
  \mathfrak s_{\xi_0}\cong\mathfrak s_{\xi_1}=: \mathfrak s.
\]
By \cref{prop:Liu-whitehead}, $Y$ is an $L$-space, so
\[
  \dim_{\F}\widehat{HF}(-Y,\mathfrak s;\F)=1.
\]
Both contact invariants are nonzero by~\textup{(i)}.  Over $\F=\Z/2\Z$, a one-dimensional vector space has a unique nonzero element, and hence
\[
  \widehat c_{\mathrm{HF}}(\xi_0;\F)
  =\widehat c_{\mathrm{HF}}(\xi_1;\F).
\]
Min--Nonino's distinction theorem in \cref{thm:MinNonino} now implies that $\xi_0$ and $\xi_1$ are contact isotopic.
\end{proof}

\begin{proof}[Proof of \cref{thm:main}]
Let
\[
  \mathscr A=\mathscr S\cup\mathscr W,
  \qquad
  \mathscr V=\mathscr F.
\]
Every member of $\mathscr A$ is an $L$-space: this is part of the definition for $\mathscr S$ and follows from \cref{prop:Liu-whitehead} for $\mathscr W$.  By \cref{thm:Matkovic}, the members of $\mathscr S$ satisfy conditions~\textup{(A1)} and~\textup{(A2)} of \cref{prop:connected-sum-closure}.  By \cref{lem:whitehead-criterion}, the same is true for the members of $\mathscr W$.  Finally, \cref{prop:s1s2-property,prop:figure-eight-property,prop:brieskorn-property} show that every member of $\mathscr F$ satisfies conditions~\textup{(V1)} and~\textup{(V2)}.

Since $\mathscr U=\mathscr A\cup\mathscr V$, \cref{prop:connected-sum-closure} applies and proves symplectization rigidity on $\mathscr U^{\#}$.  The claim that a contactomorphism induces a strict exact symplectomorphism of the symplectizations follows from \cref{lem:strict-lift}.
\end{proof}

\begin{proof}[Proof of \cref{cor:weeks}]
Min--Nonino identify $M(5,5/2)$ with the Weeks manifold \cite[Corollary~1.5]{MinNonino2026}.  Since $5/2\in\mathcal R_+$, the result follows from \cref{thm:main}.
\end{proof}

\begin{proof}[Proof of \cref{cor:hyperbolic}]
Min--Nonino state that their classification provides the first classification result on an infinite family of hyperbolic $L$-spaces \cite[Abstract and Introduction]{MinNonino2026}.  The hyperbolic members in the positive-slope range used here belong to $\mathscr W\subset\mathscr U$, so \cref{thm:main} applies.
\end{proof}

\section*{Acknowledgements}
This work was supported by JSPS KAKENHI Grant Number JP26K16990 (Grant-in-Aid for Early-Career Scientists) and by the Start-up Fund of the Kavli Institute for the Physics and Mathematics of the Universe (Kavli IPMU), The University of Tokyo.

This paper grew out of discussions between the author and ChatGPT.  The author used ChatGPT for exploratory bibliographic searches, checks of algebraic and sign computations, and English and \LaTeX{} editing.  ChatGPT was not treated as an authoritative source: the author independently checked the cited literature and all mathematical arguments, computations, and references, and takes full responsibility for the content of the manuscript.

\bibliographystyle{amsplain}
\bibliography{Symplectization_submission_final}

\end{document}